\documentclass{amsproc}

\usepackage{amssymb}
\usepackage{amsfonts}
\usepackage{amscd}
\usepackage{latexsym}
\usepackage{enumerate}
\usepackage{amsmath}
\usepackage{xy} \xyoption{all}

\newtheorem{theorem}{Theorem}[section]
\newtheorem{lemma}[theorem]{Lemma}

\theoremstyle{definition}
\newtheorem{definition}[theorem]{Definition}
\newtheorem{example}[theorem]{Example}
\newtheorem{xca}[theorem]{Exercise}

\theoremstyle{remark}
\newtheorem{remark}[theorem]{Remark}

\numberwithin{equation}{section}

\newcommand{\abs}[1]{\lvert#1\rvert}

\newcommand{\blankbox}[2]{%
  \parbox{\columnwidth}{\centering
    \setlength{\fboxsep}{0pt}%
    \fbox{\raisebox{0pt}[#2]{\hspace{#1}}}%
  }%
}

\newcounter{cs}
\stepcounter{cs}
\newcounter{ds}
\stepcounter{ds}

\newcommand{\casos}{\begin{itemize}}
\newcommand{\fcasos}{\end{itemize}\setcounter{cs}{1}}

\newcommand{\ol}{\overline}

\newcommand{\Si }{{\rm Sink}}
\newcommand{\Path }{{\rm Path}}

\DeclareMathOperator{\rL}{L}

\newtheorem{lem}{Lemma}[section]

\newtheorem{theor}[lem]{Theorem}
\newtheorem{prop}[lem]{Proposition}

\theoremstyle{definition}
\newtheorem{defi}[lem]{Definition}

\newcommand{\Z}{\mathbb{Z}}

\newcommand{\Ifree}{I_{\mathrm{free}}}
\newcommand{\Ireg}{I_{\mathrm{reg}}}

     \newcommand{\mon}[1]{\mathcal{V}(#1)}              
     
     \newcommand{\cb}[0]{K}                             

\begin{document}
\title[Leavitt path algebras over a poset of fields]{Leavitt path algebras over a poset of fields}
\author{Pere Ara}\address{Departament de Matem\`atiques, Universitat Aut\`onoma de
Barcelona, 08193, Bellaterra (Barcelona),
Spain}\email{para@mat.uab.es, mbrusten@mat.uab.es}

\thanks{Partially supported by
DGI-MINECO-FEDER through the grant MTM2017-83487-P and by the Generalitat de Catalunya through the grant 2017-SGR-1725}
\date{\today}

\begin{abstract}
Let $E$ be a finite directed graph, and let $I$ be the poset obtained as the antisymmetrization of its set of vertices with respect to a pre-order $\le$ that satisfies 
$v\le w$ whenever there exists a directed path from $w$ to $v$. 
Assuming that $I$ is a tree, we define a poset of fields over $I$ as a family
$\mathbf K = \{ K_i :i\in I \}$ of fields $K_i$ such that $K_i\subseteq K_j$ if $j\le i$. We define the concepts of a Leavitt path algebra $L_{\mathbf K} (E)$ 
and a regular algebra $Q_{\mathbf K}(E)$ over the poset of fields $\mathbf K$, and we show that $Q_{\mathbf K}(E)$ is a hereditary von Neumann regular ring, and that its monoid 
$\mon{Q_{\mathbf K}(E)}$ of isomorphism classes of finitely generated projective modules is canonically isomorphic to the graph monoid $M(E)$ of $E$.  
\end{abstract}
\maketitle

\section*{Introduction}

The work in the present paper is instrumental for the constructions
developed in \cite{ABP} (see also \cite{ABPS}). The main result in \cite{ABP} states that for every finitely generated conical
refinement monoid $M$ and every field $K$ there exists a von Neumann regular $K$-algebra $R$ such that $\mon{R}\cong M$.
The algebra $R$ is a certain universal localization of a precisely defined {\it Steinberg algebra}. The class of Steinberg algebras has been 
intensively studied in the last few years, see for instance the recent survey \cite{Rigby} and its references. It includes as prominent 
examples Leavitt path algebras of arbitrary graphs. For the proof of the main result of \cite{ABP}, one needs to interpret the algebra $R$
as an algebra obtained using certain building blocks, using constructions which were partially introduced in \cite{Atams} and fully developed in
\cite{ABP}. The purpose of this paper is to introduce these building blocks in a general way, and to show their basic properties. These algebras generalize the
main construction in \cite{AB} of the regular algebra $Q_K(E)$ of a finite directed graph $E$. The main features of the $K$-algebra $Q_K(E)$ are that it is von Neumann regular 
and that $\mon{Q_K(E)}\cong M(E)$, where $M(E)$ is the {\it graph monoid} of $E$ (see Section  \ref{sect:results} for the definition of the graph monoid).

This is generalized here as follows. 
Let $E$ be a finite graph and let $\le $ be a pre-order on $I$ such that $v\le w$ whenever there exists a directed path from $w$ to $v$.   
Let $I=E^0/{\sim}$ be the antisymmetrization of $E^0$ with respect to the pre-order $\le$, endowed with its canonical partial order, also denoted by $\le$.   
We assume throughout that $I$ is a {\it tree}, that is, that there exists a maximum element $i_0$ in $I$ and that for each $i\in I$ the interval 
$[i,i_0]$ is a chain. Under these assumptions, we define a poset of fields $\mathbf K = \{ K_i : i\in I\}$ as a collection of fields $K_i$, for $i\in I$, such that $K_i \subseteq K_j$ whenever $j\le i$.
The usual Leavitt path algebra $L_K(E)$ and regular algebra $Q_K(E)$ over a field $K$ are here generalized to a Leavitt path algebra $L_{\mathbf K}(E)$ and a regular algebra $Q_{\mathbf K}(E)$  
with coefficients in the given poset of fields $\mathbf K$. The main results state that $Q_{\mathbf K}(E)$ is a hereditary von Neumann regular ring and that $\mon{Q_{\mathbf K}(E)}\cong M(E)$ canonically.   
Our techniques here extend the ones introduced in \cite{ABmixed}, where a more restricted scope was used to study algebras over an ordered finite sequence of fields.

For more information on the realization problem we refer the reader to \cite{AB}, \cite{Atams}, \cite{AG17}, and to the survey papers \cite{directsum}, \cite{Areal}.

\section{Preliminary definitions}
\label{sect:preldefs}

All rings in this paper will be associative and all monoids will be
abelian.  A (not necessarily unital) ring $R$ is {\it von Neumann
regular} if for every $a\in R$ there is $b\in R$ such that $a=aba$.
Our basic reference for the theory of von Neumann regular rings is
\cite{vnrr}.

For a unital ring $R$, the monoid $\mon{R}$ is
the monoid of isomorphism classes of finitely generated projective
left $R$-modules, where the operation is induced by direct sum. 
If $R$ is an exchange ring (in particular, if $R$ is von Neumann
regular), then $\mon{R}$ is a conical refinement monoid, see
\cite[Corollary 1.3]{AGOP}.

Say that a subset $A$ of a poset $I$ is a {\it lower
subset} in case $q\le p$ and $p\in A$ imply $q\in A$. 
We use $I\downarrow i$ to denote the lower subset of $I$ consisting of all the elements
$j$ such that $j\le i$.

For an element $p$ of a poset $I$, write
$$\rL(p)=\rL(I ,p)=\{q\in I : q<p \text{ and } [q,p]=\{q,p\}\}.$$
An element of $\rL (p)$ is called a {\it lower cover} of $p$.

In the following, $\cb$ will denote a field and
$E=(E^0,E^1,r,s)$ a finite quiver (directed graph). 
Here $s(e)$ is the {\em source vertex} of
the arrow $e$, and $r(e)$ is the {\em range vertex} of $e$. A {\em
path} in $E$ is either an ordered sequence of arrows
$\alpha=e_1\dotsb e_n$ with $r(e_t)=s(e_{t+1})$ for $1\leq
t<n$, or a path of length $0$ corresponding to a vertex $v\in
E^0$. The paths of length $0$ are called trivial paths. A non-trivial path
$\alpha=e_1\dotsb e_n$ has length $n$ and we define
$s(\alpha)=s(e_1)$ and $r(\alpha)=r(e_n)$. We will denote the
length of a path $\alpha$ by $|\alpha|$, the set of all paths of
length $n$ by $E^n$, for $n>1$, and the set of all paths by $\text{Path} (E)$.

For $v,w\in E^0$, set $v\ge w$ in case there is a (directed) path
from $v$ to $w$. A subset $H$ of $E^0$ is called \emph{hereditary}
if $v\ge w$ and $v\in H$ imply $w\in H$. 

Let us recall the construction from \cite{AB} of the regular
algebra $Q_K(E)$ of a quiver $E$, although we will follow the
presentation in \cite{Atams} rather than the used in \cite{AB}.
That is, relations (CK1) and (CK2) below are reversed with respect
to their counterparts in \cite{AB}, so that we are led to work
primarily with {\it left} modules instead of right modules.
Therefore we recall the basic features of the regular algebra
$Q_K(E)$ in terms of the notation used here. We will only need
finite quivers in the present paper, so we restrict attention to
them.  The algebra $Q_K(E)$ fits into the following
commutative diagram of injective algebra morphisms:

\begin{equation}\notag \begin{CD}
K^{E^0} @>>> P_K(E) @>{\iota _{\Sigma}}>> P^{{\rm rat}}_K (E) @>>> P_K((E))\\
@VVV @V{\iota _{\Sigma_1}}VV @V{\iota _{\Sigma _1}}VV @V{\iota _{\Sigma _1}}VV\\
P_K(E^*) @>>> L_K(E) @>{\iota _{\Sigma}}>> Q_K(E) @>>> U_K(E) .
\end{CD} \notag \end{equation}
Here $P_K(E)$ is the path $K$-algebra of $E$, $E^*$ denotes the
inverse quiver of $E$, that is, the quiver obtained by reversing the
orientation of all the arrows in $E$, $P_K((E))$ is the algebra of
formal power series on $E$, and $P^{\rm{rat}}_K(E)$ is the algebra of rational
series, which is by definition the division closure of $P_K(E)$ in
$P_K((E))$ (which agrees with the rational closure \cite[Observation
1.18]{AB}). The maps $\iota_{\Sigma}$ and $\iota _{\Sigma_1}$
indicate universal localizations with respect to the sets $\Sigma$
and $\Sigma _1$ respectively. (We refer the reader to \cite{free} and \cite{scho} 
for the general theory of universal localization.)
Here $\Sigma$ is the set of all square
matrices over $P_K(E)$ that are sent to invertible matrices by the
augmentation map $\epsilon \colon P_K(E)\to K^{E^0}$. By
\cite[Theorem 1.20]{AB}, the algebra $P^{\rm{rat}}(E)$ coincides with the
universal localization $P_K(E)\Sigma ^{-1}$. The set
$\Sigma_1=\{\mu_v\mid v\in E^0,\,s^{-1}(v)\neq \emptyset\}$ is the
set of morphisms between finitely generated projective left
$P_K(E)$-modules defined by
 \begin{align*}
  \mu_v\colon P_K(E)v&\longrightarrow \bigoplus_{i=1}^{n_v}P_K(E)r(e^v_i)\\
  r&\longmapsto\left(re^v_1,\dotsc,re^v_{n_v}\right)
 \end{align*}
for any $v\in E^0$ such that $s^{-1}(v)\neq\emptyset$, where $s^{-1}(v)=\{ e_1^v,\dots , e_{n_v}^v \}$.
By a slight abuse of notation, we use also $\mu _v$ to denote the corresponding
maps between finitely generated projective left $P^{\rm{rat}}_K(E)$-modules and $P_K((E))$-modules
respectively.  

The following relations hold in $Q_K(E)$:

(V)\, \, \, \,\hskip.35cm $vv'=\delta _{v,v'}v$  for all
$v,v'\in E^0$.

(E1) \, \, \hskip.3cm  $s(e)e=er(e)=e$ for all $e\in E^1$.

(E2) \, \, \hskip.3cm $r(e)e^*=e^*s(e)=e^*$ for all
$e\in E^1$.

(CK1)\hskip.5cm $e^*e'=\delta _{e,e'}r(e)$ for all $e,e'\in
E^1$.

(CK2)\hskip.5cm  $v=\sum _{\{ e\in E^1\mid s(e)=v \}}ee^*$ for
every $v\in E^0$ that emits edges.

The Leavitt path algebra $L_K(E)=P_K(E)\Sigma_1^{-1}$ is the algebra
generated by $\{v\mid v\in E^0\}\cup \{e,e^*\mid e\in E ^1\}$
subject to the relations (V)--(CK2) above, see \cite{AAS}. By \cite[Theorem 4.2]{AB},
the algebra $Q_K(E)$ is a von Neumann regular hereditary ring and
$Q_K(E)=P_K(E)(\Sigma \cup \Sigma_1)^{-1}$. Here the set $\Sigma$ can be
clearly replaced with the set of all square matrices of the form
$I_n+B$ with $B\in M_n(P_K(E))$ satisfying $\epsilon (B)=0$, for all
$n\ge 1$.

\section{The results}
\label{sect:results}

Our graphs are shaped by a poset with a special property, as follows: 

\begin{defi}
 \label{poset-be-a-tree}
 Let $(I,\le )$ be a poset. We say that $I$ is a {\it tree} in case there is a maximum element $i_0\in I$ and for every
 $i\in I$ the interval $[i,i_0] : = \{ j\in I \mid i\le j\le i_0 \}$ is a chain. The element $i_0$ will be called the {\it root} of the tree $I$. 
\end{defi}

Let $\le $ be a pre-order on $E^0$ such that $w\ge v$ whenever 
there exists a path in $E$ from $w$ to $v$. Let $I$ be the
antisymmetrization of $E^0$, endowed with the partial order $\le $ induced by
the pre-order on $E^0$. Thus, denoting by $[v]$ the class of $v\in E^0$
in $I$, we have $[v]\le [w]$ if and only if $v\le w$. 

We will assume throughout this section that $E$ is a finite quiver such that $I:= E/{\sim}$ is a tree.

For $v\in E^0$, we refer to the set $[v]$ as the \emph{component}
of $v$, and we will denote by $E[v]$ the restriction of $E$ to $[v]$, that is, the graph with 
$E[v]^0= [v]$ and $E[v]^1= \{e\in E^1\mid s(e)\in [v] \text{ and } r(e)\in [v] \}$.

We will assume henceforth that we are given a family $\mathbf{K} =
\{ K_i \}_{i\in I}$ of fields such that $K_j \subseteq K_i$ if $i\le
j$. We refer to this family as a {\it poset of fields} (over $I$). 
We will define a certain $K_0$-algebra $Q_{\mathbf{K}}(E)$,
where $K_0:=K_{[v_0]}$.

Let $i,j\in I$, and suppose that $k\in (I\downarrow i)  \cap
(I\downarrow j)$. Then, since $[k,i_0]$ is  a chain, we must have
$i\le j$ or $j\le i$. Therefore, if $i$ and $j$ are incomparable
elements of $I$ then $(I\downarrow i)  \cap (I\downarrow j) =
\emptyset$. It follows that if $J$ is a lower subset of $I$, and
$x_1,\dots , x_t$ are the maximal elements of $J$, then $J=
\bigsqcup_{i=1}^t (I\downarrow x_i )$.

Given a lower subset $J$ of $I$ we consider the set of vertices
$$E^0_J = \{ v\in E^0 \mid [v]\in J \}.$$
Observe that $E^0_J$ is a hereditary subset of $E^0$. We will denote
by $E_J$ the restriction graph $E|_{E_J^0}$ corresponding to the
hereditary subset $E_J^0$ of $E^0$, and we set $p_J:= \sum_{v\in
E_J^0} v$. 

We retain the above notation. We first consider an algebra of power
series $P_{\mathbf{K}}((E))$. The algebra $P_{\mathbf{K}}((E))$ is
defined as the algebra of formal power series of the form $a=
\sum_{\gamma \in \Path (E)} a_{\gamma} \gamma $, where each
$a_{\gamma}\in K_{[r(\gamma )]}$. The usual multiplication of formal
power series gives a structure of algebra over $K_0$ on
$P_{\mathbf{K}}((E))$. Observe that if $a_{\gamma}\in K_{[r(\gamma
)]}$, $b_{\mu}\in K_{[r(\mu )]}$, and $r(\gamma ) = s(\mu)$, then
since $[r(\gamma )]\ge [r(\mu )]$, we have that
$$a_{\gamma } b_{\mu} \in K_{[r(\mu)]}=K_{[r(\gamma \mu)]},$$ because
$K_{[r(\gamma )]}\subseteq K_{[r(\mu )]}$.

The path algebra $P_{\mathbf{K}}(E)$ is defined as the subalgebra of
$P_{\mathbf{K}}((E))$ consisting of all the series in
$P_{\mathbf{K}}((E))$ having finite support.

For $i\in I$, let $\mathbf{K}_i$ be the system of fields defined
over $I\downarrow i$ by $\mathbf{K}_i = \{ K_j \}_{j\le i}$. Then
one may define the algebras $P_{\mathbf{K}_i} ((E_{I\downarrow i}))$
and $P_{\mathbf{K}_i} (E_{I\downarrow i})$.  Note that for $i\in I$
we have
$$P_{\mathbf{K}_i}(E_{I\downarrow i}) = P_{K_i}(E_{I\downarrow i}) + \Big( \bigoplus_{j\in \rL(I,i)}
P_{K_i}(E_{I\downarrow i})P_{\mathbf{K}_j}(E_{I\downarrow j})
\Big) .$$

We consider the commutative algebra $\mathcal E :=\bigoplus_{i\in
I}\bigoplus_{v\in i} K_iv$, which is a subalgebra of
$P_{\mathbf{K}}(E)$. Note that there is an augmentation map
$\epsilon \colon P_{\mathbf{K}}((E))\to \mathcal E $, and that a
square  matrix $A\in M_n (P_{\mathbf{K}}((E)))$ is invertible if and
only if $\epsilon (A)$ is invertible in $M_n(\mathcal E)$ (see e.g.
the proof of \cite[Lemma 1.8]{AB}).

Recall from Section \ref{sect:preldefs} the definition and basic properties of the algebra $P^{{\rm rat}}_K(E)$
of rational series, for a field $K$. We are now ready to define the algebra of rational series in our
setting. Let $E$ and $\mathbf{K}= \{ K_i \}_{i\in I}$ be as above. We will 
define inductively $P^{\rm rat}_J(E)$ for any lower subset $J$ of
$I$.

We first set
$$ P^{\rm rat} _{S} (E) = \bigoplus _{i\in S} P^{\rm rat}_{K_i} (E_{\{ i
\}}),$$
for any non-empty subset $S$ of minimal elements of $I$. 
Now assume that $J$ is a non-empty lower subset of $E$ and
that we have defined the algebras $P^{\rm rat}_{J''}(E)$ for all
lower subsets $J''$ of $J$. Assuming that $J\ne E^0$, let $i$ be 
a minimal element in $E^0\setminus J$. We now
define $P^{\rm rat}_{J'}(E)$, where $J' = J\cup \{i \}$, as
$$P^{\rm rat}_{J'}(E)= P^{\rm rat}_{K_{i}}(E_{I\downarrow i}) +
P^{\rm rat}_{K_{i}}(E_{I\downarrow i})P^{\rm rat}_J(E) + P^{\rm
rat}_J(E).$$ 
This defines inductively (and unambiguously) $P^{{\rm rat}}_J(E)$ for any non-empty lower subset $J$ of $I$. 
Note that if $i_1,\dots ,i_r$ are the maximal elements of $J$ then 
$$P^{\rm rat}_J(E) = \bigoplus_{k=1}^r
P^{\rm rat}_{I\downarrow i_k}(E).$$

We now define $P^{\rm rat}_{\mathbf{K}}(E):= P^{\rm rat}_{E^0}(E)$.
Observe that $P^{\rm rat}_{\mathbf{K}}(E)$ is a $K_0$-subalgebra of
$P_{\mathbf{K}}((E))$.

The following generalizes \cite[Theorem 1.20]{AB}.

\begin{theor}
\label{rat-locali} Let $E$ and $\mathbf{K}= \{ K_i \}_{i\in I}$ be
as above. Let $\Sigma$ denote the set of matrices over
$P_{\mathbf{K}}(E)$ that are sent to invertible matrices by
$\epsilon$. Then $P^{\rm rat}_{\mathbf{K}}(E)$ is the rational
closure of $P_{\mathbf{K}}(E)$ in $P_{\mathbf{K}}((E))$, and the
natural map $P_{\mathbf{K}}(E)\Sigma ^{-1}\to P^{\rm
rat}_{\mathbf{K}}(E)$ is an isomorphism.
\end{theor}

\begin{proof}
We will prove by induction that $P^{\rm rat}_{I\downarrow i} (E)$ is
the rational closure of $P_{\mathbf{K}_i} (E_i)$ in
$P_{\mathbf{K}_i}((E_i))$, where $E_i:=E_{I\downarrow i}$ is the
restriction graph of $E$ to $E^0_{I\downarrow i}$. Let $i_1,\dots
,i_r$ be the set of maximal elements of $(I\downarrow i) \setminus
\{ i \}$, and assume that the result is known for the graphs
$E_{i_k}= E_{I\downarrow i_k}$, for $k=1,\dots , r$. Changing notation we may
furthermore assume that $i=i_0$ is the maximal element of $I$.

Therefore, we have to show that 
$$S:= P_{\mathbf K}^{{\rm rat}} (E) = P_{K_0}^{\rm rat}(E)+ P_{K_0}^{\rm
rat}(E)\Big( \bigoplus _{k=1}^r P_{\mathbf{K}_{i_k}}^{\rm rat}(E_{i_k})\Big)$$ is the rational
closure of $R:=P_{\mathbf{K}}(E)$ in $P_{\mathbf{K}} (( E ))$.  Write $\mathcal R$ for this rational
closure, and recall that we are assuming that
each $P_{\mathbf{K}_{i_k}}^{\rm rat}(E_{i_k})$ is the rational closure of
$P_{\mathbf{K}_{i_k}}(E_{i_k})$ in $P_{\mathbf{K}_{i_k}}((E_{i_k}))$.

It is convenient to introduce some additional notation. Let $H$ be the hereditary subset of $E^0$ generated by $i_1,\dots , i_k$, i.e.,
$H= \bigsqcup_{k=1}^r E^0_{I\downarrow i_k}$. Recall also the notation $p_H= \sum_{v\in H} v$. Also, to somewhat simplify the
notation we set $K:=K_0$.

We start by showing that $S\subseteq \mathcal R$. Since $P^{\rm
rat}_{K}(E)$ is the rational closure of $P_K(E)$ inside $P_K((E))$ (\cite[Theorem 1.20]{AB}), we
see that $P^{\rm rat}_K(E)\subseteq \mathcal R$. Also, note that the
algebra $p_H\mathcal R=p_H\mathcal R p_H$ is rationally closed in
$p_HP_{\mathbf{K}_H} ((E_H))$ and contains $p_HP_{\mathbf{K}_H} (E_H)$, so it must contain the
rational closure of $p_HP_{\mathbf{K}_H} (E_H)$ in $p_HP_{\mathbf{K}_H} ((E_H))$ which is
$p_HP_{\mathbf{K}_H}^{\rm rat}(E_H)$ by the induction hypothesis. It follows that $P^{\rm
rat}_{K}(E)$ and $p_HP^{\rm rat}_{\mathbf{K}_H} (E_H)$ are both contained in
$\mathcal R$. Since $\mathcal R$ is a ring, we get $S\subseteq
\mathcal R$.

To show the reverse inclusion $\mathcal R \subseteq S$, take any
element $a$ in $\mathcal R$. By \cite[Theorem 7.1.2]{free}, there exist a row $\lambda\in {^nR}$, a
column $\rho\in R^n$ and a matrix $B\in M_n(R)$ with $\epsilon
(B)=0$ such that
\begin{equation}
\label{equ:5.3} a=\lambda (I-B)^{-1}\rho .
\end{equation}
Now the matrix $B$ can be written as $B=B_1+B_2$, where $B_1\in
M_n(P_{K_0}(E))\subseteq M_n(R)$ and $B_2\in M_n(R)$ satisfy that $\epsilon
(B_1)=\epsilon (B_2)=0$, all the entries of $B_1$ are supported on
paths ending in $E^0\setminus H= [v_0] $ and all the entries of $B_2$ are
supported on paths ending in $H$. Note that, since $H$ is
hereditary, this implies that all the paths in the support of the
entries of $B_1$ start in $E^0\setminus H$ and thus $B_2B_1=0$. It
follows that
\begin{equation}
\label{Binverse}
(I-B)^{-1}=(I-B_1-B_2)^{-1}=(I-B_1)^{-1}(I-B_2)^{-1},
\end{equation}
and therefore
$(I-B)^{-1}=(I-B_1)^{-1}+(I-B_1)^{-1}B_2(I-B_2)^{-1}\in M_n(S).$ It
follows from (\ref{equ:5.3}) that $a\in S$, as desired. 

Since the set $\Sigma$ is precisely the set of square matrices over
$R$ which are invertible over $P_{\mathbf{K}}((E))$, we get from a well-known
general result (see for instance \cite[Lemma 10.35(3)]{Luck}) that
there is a surjective $K$-algebra homomorphism $\phi \colon
R{\Sigma}^{-1}\to \mathcal R$.

The rest of the proof is devoted to show that $\phi $ is injective.
We have a commutative diagram

\begin{equation}
\begin{CD}
P_K(E)\Sigma(\epsilon_K )^{-1} @>>> R\Sigma^ {-1}  \\
@V{\phi_K}V{\cong}V  @V{\phi}VV    \\
P_K^{\text{rat}}(E) @>>> \mathcal R
\end{CD}
\end{equation}
where the map $\phi_K$ is an isomorphism by \cite[Theorem 1.20]{AB}.
The map $P_K(E)\Sigma(\epsilon_K )^{-1} \linebreak
\to  \mathcal R $ is injective, so the map $P_K(E)\Sigma(\epsilon_K )^{-1}
\to R\Sigma^ {-1}  $ must also be injective. Hence the
$K$-subalgebra of $R\Sigma^{-1}$ generated by $P_K(E)$ and the
entries of the inverses of matrices in $\Sigma (\epsilon_K )$ is
isomorphic to $P_K^{\text{rat}}(E)$. Observe that we can replace
$\Sigma$ by the set of matrices of the form $I-B$, where $B$ is a
square matrix over $R$ with $\epsilon (B)=0$. 
An element $x$ in
$R\Sigma^{-1}$ is of the form
\begin{equation}
\label{can-form} x=\lambda (I-B)^{-1} \rho
\end{equation}
with $\lambda \in {^n R}$ and $\rho \in R^ n$,  and $B\in M_n(R)$ satisfies $\epsilon
(B)=0$.

\medskip

\noindent {\it Claim 1.}  We have
$$p_HR\Sigma ^{-1} \cong p_HP_{\mathbf{K}_H}^{\text{rat}}(E_H)\cong  P_{\mathbf{K}_H}(E_H)\Sigma (\epsilon _{\mathbf{K}_H})^{-1} .$$

\noindent {\it Proof of Claim 1.}
Observe first that we have a natural homomorphism $$\psi \colon P_{\mathbf{K}_H}(E_H)\Sigma (\epsilon _{\mathbf{K}_H})^{-1}\longrightarrow p_HR\Sigma^{-1}.$$
By the induction hypothesis, the composition of $\psi$ with the map $\phi$ is injective, because
it coincides with the canonical map from $P_{\mathbf{K}_H}(E_H)\Sigma (\epsilon_{\mathbf{K}_H})^{-1}$ onto \linebreak
$p_HP_{\mathbf{K}_H}^{\text{rat}}(E_H)\cong
P_{\mathbf{K}_H}^{\text{rat}}(E_H)$. Hence $\phi$ induces an isomorphism from 
$$\mathfrak S := \psi (P_{\mathbf{K}_H}(E_H)\Sigma (\epsilon _{\mathbf{K}_H}^H)^{-1})$$
onto $p_HP_{\mathbf{K}_H}^{\text{rat}}(E_H)$. Note that $\mathfrak S$ is precisely the subalgebra of
$p_HR\Sigma^{-1}$ generated by $p_HP_{\mathbf{K}_H}(E_H)$ and the entries of the
inverses of matrices of the form $p_HI-B$, with $B$ a square matrix
over $p_HP_{\mathbf{K}_H}(E_H)$  with $\epsilon (B)=0$. For an element $x$ in
$R\Sigma ^{-1}$, we write it in its canonical form (\ref{can-form})
and we write $B=B_1+B_2$ with all the entries in $B_1$ ending in
$E^0\setminus H$ and all the entries of $B_2$ ending in $H$.

Now multiply (\ref{can-form}) on the left by $p_H$ and use (\ref{Binverse}) to get
\begin{align*}
p_Hx & =p_H\lambda (I-B_1)^{-1}\rho +p_H\lambda (I-B_1)^{ -1}B_2(I-B_2)^{-1}\rho   \\
& = p_H\lambda p_H(I-B_1)^{-1}\rho + p_H \lambda p_H(I-B_1)^{ -1}B_2(I-B_2)^{-1}\rho \\
& = p_H \lambda p_H \rho + p_H \lambda p_H B_2 p_H (I-B_2)^{-1} \rho.
\end{align*}

 Write $B_2=B_2'+B_2''$, where all the entries of $B_2'$ start  in $E^0\setminus H$ and all the entries in $B_2''$ start
 in $H$ (and so end in $H$ as well).
 Note that $(I-B_2')^{-1}=I+B_2'$, because $B_2'^ 2=0$, so that $p_H(I-B_2')^{-1} =p_H$. Since $B_2''B_2'=0$
 we have  $(I-B_2)^{-1}=(I-B_2')^{-1}(I-B_2'')^{-1}$,
 and thus
 $$p_Hx=p_H\lambda p_H\rho p_H + p_H\lambda p_H B_2 p_H (I-B_2'')^{-1} p_H\rho p_H .$$
 It follows that $p_Hx\in \mathfrak S$, and so $p_HR\Sigma^{-1} = \mathfrak S$, as wanted. \qed

 Assume now that $x\in \ker (R\Sigma ^{-1} \to \mathcal R )= \ker ( R\Sigma ^{-1} \to
 P_{\mathbf{K}}^{{\rm rat}}(E))$ and write $x$ as in (\ref{can-form}), with $B=B_1+B_2$ as before.
 Then
\begin{equation}
\label{express1} x=\lambda (I-B_1)^{-1}\rho + \lambda (I-B_1)^{-1}
B_2(I-B_2)^{-1} \rho .
\end{equation}
Multiplying on the right by $1-p_H$, we get
 $$x(1-p_H)=\lambda (I-B_1)^{-1}\rho (1-p_H) = \lambda (1-p_H)(I-B_1)^{-1}\rho (1-p_H)  \in P_K(E)\Sigma (\epsilon _K)^{-1}$$
 and $0=\phi (x(1-p_H))=\phi_K(x(1-p_H))$.  Since $\phi _K$ is an isomorphism, we get $x(1-p_H)=0$.

Hence we have
\begin{equation}
\label{express2} x=  \lambda (I-B_1)^{ -1} \rho_2 +  \lambda
(I-B_1)^{-1}B_2(I-B_2)^{-1}\rho_2   ,
\end{equation}
 where $\rho=\rho_1+\rho_2$ with $\rho_1$ ending in $E^0\setminus H$ and $\rho_2$ ending in $H$.
 By Claim 1 we have $p_Hx=0$, because $\phi $ is an isomorphism when restricted to $p_HR\Sigma^{-1}$.
 Now we are going to find a suitable expression for $x=(1-p_H)xp_H$.
 Write $\lambda =\lambda _1+\lambda _2$ with $\lambda _1=(1-p_H)\lambda $ and $\lambda _2=p_H\lambda $.
  Then
  \begin{equation}
 \label{express3}
(1-p_H) \lambda (I-B_1)^{-1}\rho_2 =\lambda _1 (I-B_1)^{ -1}\rho _2
.
 \end{equation}
 Similarly $(1-p_H)\lambda (I-B_1)^{-1}B_2(I-B_2)^{-1}\rho_2 =\lambda _1 (I-B_1)^{-1} B_2(I-B_2)^{-1}\rho_2 $.
 Write $B_2=B_2'+B_2''$, with $B_2'$ starting in $E^0\setminus H$ and $B_2''$ starting in $H$. Then $B_2''B_2'=0$ and $(I-B_2)^{-1}=(I-B_2')^{-1}(I-B_2'')^{-1}$,
 so that
  \begin{align}
 \label{express4}
\notag &  (1-p_H)\lambda (I-B_1)^{-1}B_2(I-B_2)^{-1}\rho_2 =\lambda _1 (I-B_1)^{-1} B_2(I-B_2)^{-1}\rho_2 \\
  &  = \lambda_1 (I-B_1)^{-1}B_2(I+B_2')(I-B_2'')^{-1}\rho_2 \\
 \notag & =  \lambda_1 (I-B_1)^{-1}B_2(I-B_2'')^{-1}\rho_2.
   \end{align}
From (\ref{express3}),  (\ref{express4}) and (\ref{express2}) we get
   \begin{equation}
\label{express5} x=  (1-p_H)xp_H=  \lambda_1 (I-B_1)^{ -1} \rho_2+
\lambda_1(I-B_1)^{-1}B_2  (I-B_2'')^{-1}\rho_2   .
\end{equation}
It follows that $x\in \sum _{i=1}^k P_K^{\text{rat}}
(E/H)e_iP_{\mathbf{K}_H}^{\text{rat}} (E_H)$, where $e_1,\dots ,e_k$ is the
family of {\it crossing edges}, that is, the family of edges $e\in
E^1$ such that $s(e)\in E^0\setminus H$ and $r(e)\in H$. Write
$x=\sum _{i=1}^k \sum _{j=1}^{m_i} a_{ij}e_ib_{ij} $ for certain
$a_{ij}\in P_K^{\text{rat}} (E/H)$ and $b_{ij}\in P_{\mathbf{K}_H}^{\text{rat}}
(E_H)$. Then we have
$$0=\phi (x)= \sum _{i=1}^k \sum _{j=1}^{m_i} a_{ij}e_ib_{ij} ,$$
this element being now in $P_{\mathbf{K}}((E))$. Clearly this implies that $
\sum _{j=1}^{m_i} a_{ij}e_ib_{ij}=0$ in $P_{\mathbf{K}}((E))$ for all $i=
1,\dots ,k$. So the result follows from the following claim:

\medskip

\noindent  {\it Claim 2.}  Let $e$ be a crossing edge, so that
$s(e)\in E^0\setminus H$ and $r(e)\in H$. Assume that $b_1,\dots
,b_m\in r(e) P_{\mathbf{K}_H}((E_H))$ are $K$-linearly independent elements,
and assume that $a_1e,\dots , a_me$ are not all $0$, where
$a_1,\dots ,a_m\in P_K((E\setminus H))$. Then $\sum _{i=1}^m a_ieb_i
\ne 0$ in $P_{\mathbf{K}}((E))$.

\noindent {\it Proof of Claim 2.} By way of contradiction, suppose
that $\sum _{i=1}^m a_ieb_i = 0$. We may assume that $a_1e\ne 0$.
Let $\gamma $ be a path in the support of $a_1$ such that $r(\gamma
)=s(e)$. For every path $\mu$ with $s(\mu )=r(e)$ we have that the
coefficient of $\gamma e \mu $ in $a_ieb_i$ is $a_i(\gamma)
b_i(\mu)$, so that $\sum _{i=1}^m a_i(\gamma )b_i(\mu) =0$ for every
$\mu$ such that $s(\mu)=r(e)$. Since every path in the support of
each $b_i$ starts with $r(e)$, we get that
$$\sum_{i=1}^m a_i(\gamma )b_i=0$$
with $a_1(\gamma)\ne 0$, which contradicts the linear independence
over $K$ of $b_1,\dots, b_m$. \qed

This concludes the proof of the theorem.
\end{proof}

Following \cite[Section 2]{AB}, we define, for $e\in E^1$, the
right transduction $\tilde{\delta}_e\colon P_{\mathbf{K}}((E))\to P_{\mathbf{K}}((E))$
corresponding to $e$ by $$\tilde{\delta_e} (\sum_{\alpha\in
E^*}\lambda_\alpha \alpha)=\sum_{\substack{\alpha\in
E^*\\s(\alpha)=r(e)}} \lambda_{e\alpha } \alpha.$$ Note that the left transductions are not
even defined on $P_{\mathbf{K}}((E))$ in general.

Observe that $R:=P_{\mathbf{K}}(E)$ is closed under
all the right transductions, i.e.  $\tilde{\delta}_e(R)\subseteq R$.
Some of
the proofs in \cite{AB} make use of the fact that the usual path
algebra $P_K(E)$ is closed under {\it left and right}
transductions. Fortunately we have been able to overcome the
potential problems arising from the failure of invariance of $R$
under left transductions by using alternative arguments.

We are now ready to get a description of the algebra
$Q_{\mathbf{K}}(E)$ as a universal localization of
the path algebra $P_{\mathbf{K}}(E)$.

 Write $R:=P_{\mathbf{K}}(E)$. For any $v\in E^0$ such that $s^{-1}(v)\neq\emptyset$ we put
 $s^{-1}(v)=\{e^v_1,\dotsc,e^v_{n_v}\}$, and we
consider the left $R$-module homomorphism
 \begin{align*}
  \mu_v\colon Rv&\longrightarrow \bigoplus_{i=1}^{n_v}Rr(e^v_i)\\
  r&\longmapsto\left(re^v_1,\dotsc,re^v_{n_v}\right)
 \end{align*}
 Write $\Sigma_1=\{\mu_v\mid v\in E^0,\,s^{-1}(v)\neq \emptyset\}$.

\begin{theor}
\label{hereditarymixed} Let $\mathbf{K}$ and $E$ be as before. Let $\Sigma$ denote
the set of matrices over $P_{\mathbf{K}}(E)$ that are
sent to invertible matrices by $\epsilon$ and let $\Sigma _1$ be the
set of maps defined above. Then we have that
$Q_{\mathbf{K}}(E):=
P_{\mathbf{K}}(E)(\Sigma \cup \Sigma _1)^{-1}$ is a
hereditary von
Neumann regular ring and all finitely generated projective
$Q_{\mathbf{K}}(E)$-modules are induced from $P^{{\rm
rat}}_{\mathbf{K}}(E)$. Moreover each element of $Q_{\mathbf{K}}(E)$ can be written as a
finite sum $$\sum _{\gamma \in \mathrm{Path}(E)} a_{\gamma }\gamma^*,$$
where $a_{\gamma}\in P_{\mathbf{K}}^{\rm rat} (E) r(\gamma )$.
\end{theor}

\begin{proof} First observe that $P_{\mathbf{K}}(E)$
is a hereditary ring and that
$\mon{P_{\mathbf{K}}(E)}=(\mathbb Z^+)^d$, where
$|E^0|=d$. This follows by successive use of \cite[Theorem
5.3]{Bergman}.           

In order to get that the right transduction $\tilde{\delta}_e\colon
P_{\mathbf{K}}((E))\to P_{\mathbf{K}}((E))$ corresponding to $e$ is a right
$\tau_e$-derivation, that is,
\begin{equation}
\label{right-der} \tilde{\delta}_e
(rs)=\tilde{\delta}_e(r)s+\tau_e(r)\tilde{\delta}_e(s)
\end{equation}
for all $r,s\in P_{\mathbf{K}}((E))$, we have to modify
slightly the definition of $\tau _e$ given in \cite[page 220]{AB}.
Concretely we define $\tau _e$ as the endomorphism of $P_{\mathbf{K}}((E))$
given by the composition
$$P_{\mathbf{K}}((E))\to \prod_{v\in E^0} K_{[v]}v\to \prod_{v\in E^0} K_{[v]}v\to P_{\mathbf{K}}((E)),$$
where the first map is the augmentation homomorphism, the third map is the canonical inclusion,
and the middle map is the $K_0$-lineal map
given by sending $s(e)$ to $r(e)$, and any other idempotent
$v$ with $v\ne s(e)$ to $0$. Observe that this restricts to an
endomorphism of $P_{\mathbf{K}}(E)$ and that the
proof in \cite[Lemma 2.4]{AB} gives the desired formula
(\ref{right-der}) for $r,s\in P_{\mathbf{K}}((E))$ and, in particular for
$r,s\in P_{\mathbf{K}}(E)$. 

Note that, by the argument given after the proof of Proposition 2.7 in \cite{AB}, the algebras
$P^{\rm rat}_{\mathbf{K}}(E)$ are stable under all the right transductions. Hence, the constructions in
\cite[Section 2]{AB} apply to $R:=P^{\rm
rat}_{\mathbf{K}}(E)$ (with some minor changes), and
we get that $R\Sigma_1^{-1}=R\langle \ol{E};
\tau,\tilde{\delta}\rangle /I$, where $I$ is the ideal of $R\langle
\ol{E}; \tau,\tilde{\delta}\rangle$ generated by the idempotents
$q_v:=v-\sum _{e\in s^{-1}(v)}ee^*$ for $v\notin \Si (E)$.

By Theorem \ref{rat-locali}, we have
\begin{align*}
Q_{\mathbf{K}}(E) & = P_{\mathbf{K}}(E)(\Sigma \cup \Sigma_1)^{-1}
= (P_{\mathbf{K}}(E)\Sigma^{-1}) \Sigma_1^{-1} \\ 
 & = P^{{\rm rat}}_{\mathbf{K}} (E) \Sigma_1^{-1}
= (P^{{\rm rat}}_{\mathbf{K}}(E))\langle \ol{E};\tau ,\tilde{\delta}\rangle/I.
 \end{align*}
By a result of Bergman and Dicks \cite{BD} any universal localization
of a hereditary ring is hereditary, thus we get that both
$P^{\text{rat}}_{\mathbf{K}}(E)$ and
$Q_{\mathbf{K}}(E)$ are hereditary rings. Since
$P^{\text{rat}}_{\mathbf{K}}(E)$ is hereditary,
closed under inversion in $P_{\mathbf{K}}((E))$ (by Theorem
\ref{rat-locali}), and closed under all the right transductions
$\tilde{\delta}_e$, for $e\in E^1$, the proof of \cite[Theorem
2.16]{AB} gives that $Q_{\mathbf{K}}(E)$ is von
Neumann regular and that every finitely generated projective module is
induced from $P^{{\rm rat}}_{\mathbf{K}}(E)$.

The last statement follows from \cite[Proposition 2.5(ii)]{AB}. This concludes the proof of the theorem.
\end{proof}

\begin{remark}
\label{no-invneed} Theorem 2.16 in \cite{AB} is stated for a
subalgebra $R$ of $P_K((E))$ which is closed under all left and
right transductions (and which is inversion closed in $P_K((E))$).
However the invariance under right transductions is only used in the
proof of that result to ensure that the ring $R$ is left
semihereditary. Since we are using the opposite notation concerning
(CK1) and (CK2), the above hypothesis translates in our setting into
the condition that $P_{\mathbf{K}}(E)$ and $P^{{\rm
rat}}_{\mathbf{K}}(E)$ should be invariant under all
{\it left} transductions, which is not true in general as we
observed above. We overcome this problem by the use of the result of
Bergman and Dicks (\cite{BD}), which guarantees that $
P_{\mathbf{K}}(E)$  and $P^{{\rm
rat}}_{\mathbf{K}}(E)$ are indeed right and left
hereditary (see the proof of Theorem \ref{hereditarymixed}).
\end{remark}

For our last result, we need an auxiliary lemma.

\begin{lem}
\label{lem:Poset-system}
Let $(I,\le )$ be a finite poset which is a tree, with maximum element $i_0$.
Let $\mathbf{K}=\{ K_i \}_{i\in I}$ be a poset
of fields, so that $K_i\subseteq K_j$ if $j\le i $ in $I$. Then there exists a field $K$ and an embedding of
$\mathbf{K}$ into the constant system $K$ over $I$, i.e., there is a collection of field morphisms $\varphi _i\colon K_i\to K$
such that $\varphi _i|_{K_j} = \varphi _j$ for all $i\le j$.
\end{lem}

\begin{proof}
For a finite poset $I$ as in the statement, define the {\it depth} of $I$ as the maximum of the lengths of maximal chains
$i_k < i_{k-1} < \cdots < i_0 $.  We prove the result by induction on the depth of $I$. If the depth of $I$ is $0$, there is
nothing to prove. Assume the result is true for finite posets of depth at most $r$ satisfying the hypothesis in the statement and let $(I,\le )$
such a finite poset of depth $r+1$, where $r\ge 0$. Write $\rL (I,i_0)= \{ a_1,\dots , a_n \}$. Then the depth of each of the posets
$I\downarrow a_t$ is at most $r$, so that
there are embeddings $\psi _t\colon \mathbf{K}_{I\downarrow a_t} \to L_t$ for some fields $L_t$, for $t=1,\dots , r$.
This means that there are field embeddings $\psi _{t,i}\colon K_i \to L_t$ for each $i\le a_t$, such that $\psi _{t,i}|_{K_j}=
\psi _{t,j}$ for $i\le j \le a_t$. Now, setting $K_0:= K_{i_0}$, there exist a field $K$ and embeddings
$\delta_t \colon  L_t \to K$ for $t=1,\dots ,n$ such that $\delta _t\circ ((\psi_{t,a_t})|_{K_0}) = \delta _{t'}\circ (\psi_{t',a_{t'}})|_{K_0})$ for $1\le t,t'\le n$.
Define $\varphi _i\colon K_i \to K$ to be $\delta_{t}\circ \psi _{t,i}$ if $i\le a_t$.
(Observe that this is well-defined by our hypothesis
on $I$.) Finally define $\varphi _{i_0} \colon K_0\to K$ as $\varphi_{i_0}:= \delta _t\circ ((\psi_{t,a_t})|_{K_0})$, which is independent of
the choice of $t$ by the above.

If $i\le j \le a_t$ for some $t$, then
$$\varphi_i|_{K_j} = (\delta_t \circ \psi_{t,i})|_{K_j} =  \delta_t \circ (\psi_{t,i}|_{K_j})= \delta_t\circ \psi_{t,j} = \varphi _j.$$
If $i<i_0$, then there is a unique $t\in \{1,\dots , n\}$ such that $i \le a_t < i_0$, and so
$$\varphi _i|_{K_0} = (\delta_t \circ \psi_{t,i})|_{K_0} = \delta _t \circ (\psi_{t,i}|_{K_{a_t}})|_{K_0} =
\delta_t \circ (\psi_{t,a_t})|_{K_0} = \varphi _{i_0} .$$
This completes the proof.
\end{proof}

Define the Leavitt path algebra
$L_{\mathbf{K}}(E)$ associated to the poset of fields $\mathbf K$ and the finite quiver $E$ as the universal localization of
$P_{\mathbf{K}}(E)$ with respect to the set
$\Sigma_1$. Let $M(E)$ be the abelian monoid with generators $E^0$
and relations given by $v=\sum _{e\in s^{-1}(v)} r(e)$, see
\cite{AMP} and \cite[Chapter 3]{AAS}.

\begin{theor}
\label{isomorphs} With the above notation, we have natural
isomorphisms
$$M(E)\cong  \mon{L_{\mathbf{K}}(E)}\cong
\mon{Q_{\mathbf{K}}(E)}.$$
\end{theor}

\begin{proof}
The proof that $M(E)\cong \mon{L_{\mathbf{K}}(E)}$
follows from Bergman's results \cite{Bergman}, as in
\cite[Theorem 3.5]{AMP}.

Note that $R:=P^{\rm rat}_{\mathbf{K}}(E)$ is
semiperfect. Thus we get $\mon{R}\cong (\Z^+)^{|E_0|}$ in the
natural way, that is the generators of $\mon{R}$ correspond to the
projective modules $Rv$ for $v\in E^0$. By Theorem
\ref{hereditarymixed}, we get that the natural map $M(E)\to
\mon{Q_{\mathbf{K}}(E)}$ is surjective. To show
injectivity, let $K$ and $\{ \varphi_i \}_{i\in I}$ be a field and maps satisfying the hypothesis of Lemma \ref{lem:Poset-system}.
Then it is easily seen that there is
a unital ring homomorphism $\varphi \colon Q_{\mathbf{K}}(E) \to Q_K(E)$. Indeed, one directly verifies that there
is well-defined homomorphism  $\varphi \colon P_{\mathbf{K}}(E)\to P_K(E)$ given by the rule
$$\varphi (\sum a_{\gamma} \gamma ) = \sum \varphi _{[r(\gamma )]}(a_{\gamma}) \gamma .$$
Let $\Sigma $ (respectively $\Sigma_K$) be the set of all square matrices over $P_{\mathbf{K}}(E)$ (respectively over $P_K(E)$)
which are sent to invertible matrices through the augmentation map $\epsilon$. Since $\epsilon (\varphi (A))= \varphi (\epsilon (A))$, it is obvious that
$\varphi (\Sigma ) \subseteq \Sigma_K$, and consequently the map $\varphi$ can be uniquely extended to a homomorphism (also denoted by $\varphi$)
from $Q_{\mathbf{K}}(E) = P_{\mathbf{K}}(E)(\Sigma \cup \Sigma_1)^{-1}$ to $Q_K(E)=P_K(E)(\Sigma_K\cup \Sigma_1)^{-1}$.
By \cite[Theorem 4.2]{AB}, we have $M(E)\cong \mon{Q_{K_0}(E)}$ and $M(E)\cong \mon{Q_K(E)}$ canonically, so that we get  
\begin{equation*}
 M(E)\cong \mon{Q_{K_0}(E)}\longrightarrow
\mon{Q_{\mathbf{K}}(E)}\longrightarrow
\mon{Q_{K}(E)}\cong M(E) ,
\end{equation*}
and  the composition of the maps above is the identity. It
follows that the map $M(E)\to
\mon{Q_{\mathbf{K}}(E)}$ is injective and so it
must be a monoid isomorphism.
\end{proof}

We adopt now the specific setting of \cite{ABP} for our graphs in order to obtain a property of $Q_{\mathbf K}(E)$ needed in \cite{ABP}. 
So in addition to our previous assumptions on $E$, we make the following requirements:

\begin{enumerate}
\item There is a partition $I=I_{{\rm free}}\bigsqcup I_{{\rm reg}}$ (where $I=E/{\sim}$). 
\item For each $v\in E^0$ such that $[v]$ is not minimal in $I$ and
$[v]\in I_{{\rm free}}$, we have that $E[v]$ is a graph with a single vertex $v$ and a single arrow $\alpha^v$,
with $s(\alpha^v)= r(\alpha^v)= v$.
\item For each $v\in E^0$ such that $[v]\in I_{{\rm reg}}$, we have that $E[v]$ is a graph with at least two edges.
\item  If $[v]$ is minimal in $I$, then either $v$ is a sink or
$[v]\in I_{{\rm reg}}$.
\end{enumerate}

With these conditions at hand, we finally observe that we can obtain 
$Q_{\mathbf{K}}(E)$ by inverting a smaller set of matrices over $P_{\mathbf{K}}(E)$.
For each non-minimal $[v]\in \Ifree$, let
$$\Sigma_i := \{ p(\alpha^v) \mid p(x) \in K_{[v]}[x] \text{ and } p(0)\ne 0 \}.$$
be the set of polynomials in $\alpha^v$ with coefficients in $K_{[v]}$ with nonzero constant term.
Set $p_i =\sum_{v\in i} v$ for $i\in I$. 
For each $i\in \Ireg$, let $\Sigma_i$ be the set of square matrices $M$ over
$p_iP_{K_i}(E)p_i$ such that $\epsilon_i(M)$ is an invertible matrix over $\prod_{v\in i} K_i$, where $p_iP_{K_i}(E)p_i\to
\prod_{v\in i} K_i$ is the augmentation homomorphism.
Finally set $\Sigma ' = \bigsqcup_{i\in I} \Sigma _i'$, where, for $i\in I$,  
$\Sigma_i'= \{ M+(1-p_i)I_m \mid M \in M_m(P_{\mathbf K} (E))\cap \Sigma_i \}$.

\begin{prop}
\label{prop:enough-inversion}
Let $E$ and $\Sigma, \Sigma'$ be as above. Then we have $Q_{\mathbf{K}}(E)= L_{\mathbf{K}}(E)(\Sigma')^{-1}$.
\end{prop}

\begin{proof}
Since $Q_{\mathbf{K}}(E)= L_{\mathbf{K}}(E)\Sigma^{-1}$, we only have to show that every element of $\Sigma$ is invertible over
$L_{\mathbf{K}}(E)(\Sigma')^{-1}$. For this, it is enough to show that any matrix of the form $I-A$ is invertible over $L_{\mathbf{K}}(E)(\Sigma')^{-1}$,
where $A$ is a square matrix over $P_{\mathbf{K}}(E)$ such that $\epsilon (A)= 0$.

Now given such a matrix $A$, we can uniquely write it in the form
$$A= A_0+B+A_1+\dots + A_l,$$
where $\rL (I,i_0)=\{i_1,\dots, i_l \}$, $\epsilon (B)= \epsilon (A_k)=0$ for $k=0,\dots , l$,
$A_0$ is a square matrix over $p_{i_0}P_{\mathbf{K}}(E)p_{i_0}$, all paths in the support of $B$ start in a
vertex in the component of $i_0$ and end in a vertex in the component
of some $i\in I$, where $i<i_0$, and each $A_k$ is a square matrix over $P_{\mathbf{K}_{i_k}}(E_{i_k})$.
Here, $E_{i_k}$ is the hereditary subset $E_{I\downarrow i_k}$ and $\mathbf{K}_{i_k}$ is the restriction
of the poset of fields $\mathbf{K}$ to $I\downarrow i_k$.

Observe that $(B+\sum_{k=1}^l A_k)A_0= 0$, so that we obtain
$$(I-A)^{-1} = (I-A_0)^{-1}(I-(B+\sum_{k=1}^l A_k))^{-1}$$
in any ring where the matrices on the right hand side are invertible.
Now observe that $(\sum_{k=1}^l A_k)B=0$, because all the paths in the support of $B$ start at the component $i_0$, so we get
$$(I- (B+ \sum_{k=1}^l A_k))^{-1} = (I-B)^{-1}(I-\sum_{k=1}^l A_k)^{-1}$$
in any ring where the matrices on the right hand side are invertible.
Now since $$(I\downarrow i_k)\cap (I\downarrow i_{k'}) =\emptyset$$ for $k\ne k'$  we get
that
$$(I- \sum_{k=1}^l A_k)^{-1}= \prod_{k=1}^l (I-A_k)^{-1}$$
again in any ring where all matrices on the right hand side are invertible.

Continuing in this way, we obtain that the matrix $(I-A)^{-1}$ can be expressed as a product of matrices
of the forms $(I-A')^{-1}$, where $A'$ is a matrix over $p_iP_{K_i}(E)p_i$ such that $\epsilon (A')=0$, and $(I-B')^{-1}$, where $B'$ is a matrix
such that all paths in the support of $B'$ start at a vertex in a fixed component $i\in I$  and end at a component which is strictly
less than $i$. Now observe that $(B')^2= 0$ and so $(I-B')^{-1}= I+B'$ is a matrix over $P_{\mathbf{K}}(E)$. Therefore, since all the matrices
$I-A'$ and $I-B'$ as above are invertible in $L_{\mathbf{K}}(E)(\Sigma')^{-1}$, we see that $I-A$ is also invertible, and the proof is complete.
(Note that, in case $i=[v]\in \Ifree$ is non-minimal, the determinant of the matrix $Iv-A'$ is of the form $v+p(\alpha^v)$, where $p(0)=0$,
so the matrix $I-A'$ is invertible over $L_{\mathbf{K}}(E)(\Sigma')^{-1}$.)
\end{proof}

\end{document}